\documentclass[12pt]{article}
  \usepackage{inputenc}
  \usepackage[english]{babel}
  \usepackage{amsfonts,amssymb,amsmath,amstext,amsbsy,amsopn,amscd,graphicx}
  \usepackage{graphics}
  \usepackage{amsthm}

 \newtheorem{thm}{Theorem}[section]

  \newtheorem{lem}{Lemma}[section]
  \newtheorem{cor}{Corollary}[section]
  \newtheorem{dfn}{Definition}[section]

\title{On a Generalization of Alexander Polynomial for Long Virtual Knots}

\author{Denis Afanasiev}

\date{}

\begin{document}

\newcommand{\skcrro}{\raisebox{-0.25\height}{\includegraphics[width=0.5cm]{skcrro.eps}}}
\newcommand{\skcrlo}{\raisebox{-0.25\height}{\includegraphics[width=0.5cm]{skcrlo.eps}}}

\maketitle

\begin{abstract}

We construct new invariant polynomial for long virtual knots. It is
a generalization of Alexander polynomial. We designate it by $\zeta$
meaning an analogy with $\zeta$-polynomial for virtual links. A
degree of $\zeta$-polynomial estimates a virtual crossing number. We
describe some application of $\zeta$-polynomial for the study of
minimal long virtual diagrams with respect number of virtual
crossings.
\end{abstract}

\setcounter{section}{1}

Virtual knot theory was invented by Kauffman around 1996 \cite{KaV}.
Long virtual knot theory was invented in \cite{GPV} by M.~Goussarov,
M.~Polyak, and O.~Viro. $\zeta$-polynomial for virtual link was
introduced independently by several authors (see
\cite{KR},\cite{Saw},\cite{SW},\cite{MaXi}), for the proof of their
coincidence, see \cite{BF}. The idea of two types of classical
crossings in a long diagram, which were called  $\circ$ (circle) and
$\ast$ (star), was invented by V.O.~Manturov
 (see \cite{MaL},\cite{MyBook}). In present paper we called $\circ$
 and $\ast$ crossings by {\em early overcrossing} and {\em early
 undercrossing} respectively. To consider early overcrossings and early
  undercrossings is the basis idea for a construction of $\zeta$-polynomial
  in the case of long virtual knots.

\begin {dfn}
{\upshape

 By a {\em long virtual knot diagram} we mean a smooth
immersion $f:\mathbb{R}\rightarrow\mathbb{R}^2$ such that:

\noindent 1) outside some big circle, we have $f(t)=(t,0)$;

\noindent 2) each intersection point is double and transverse;

\noindent 3) each intersection point is endowed with classical (with
a choice for underpass and overpass specified) or virtual crossing
structure. }
\end {dfn}

 \begin {dfn}
 {\upshape
  A {\em long virtual knot} is an equivalence class of long virtual
  knot diagrams modulo generalized Reidemeister moves.
  }
 \end {dfn}

\begin {dfn}
{\upshape By an {\em arc} of a long virtual knot diagram we mean a
connected component of the set, obtained from the diagram by
deleting all virtual crossings (at classical crossing the
undercrossing pair of edges of the diagram is thought to be disjoint
as it is usually illustrated). }
\end {dfn}

\begin {dfn}
{\upshape
 We say that two arcs $a, a'$ belong to the same {\em long
arc} if there exists a sequence of arcs $a = a_{1},\dots,  a_{n+1} =
a'$ and virtual crossings $c_{1},\dots,c_{n}$ such that for $i =
1,\dots, n$ the arcs $a_{i},a_{i+1}$ are incident to $c_i$ from
opposite sides. }
\end {dfn}

Throughout the paper, we mean that initial and final long arcs,
${\gamma}_{-}$ and ${\gamma}_{+}$, form united long arc
$\gamma={\gamma}_{-}\cup{\gamma}_{+}$. Let $D$ be a long virtual
diagram with $n\geqslant 1$ classical crossings. Hence, there is a
natural pairing of all classical crossings and all long arcs:
classical crossing $v$ and long arc $\gamma$, which emanates from
$v$, are paired.

We say that  classical crossing $v$ is {\em early overcrossing}
({\em early undercrossing}) if we have an arc passing over (under)
$v$ at first, in the natural order on long virtual diagram (see also
\cite{KM}, p. 139).

\begin {dfn}
 {\upshape
 An {\em incidence coefficient}
 $[v:a]\in T=\mathbb{Z}[p,p^{-1},q,q^{-1}]/((p-1)(p-q),(q-1)(p-q))$ of classical
crossing $v$ and arc $a$ is defined as a sum of some of three
polynomials: $[v:a]={\varepsilon}_1 1+{\varepsilon}_2
(t^{sgn\,v}-1)+{\varepsilon}_3(-t^{sgn\,v})$, where
${\varepsilon}_i\in\{0,1\},i=1,2,3$; $t=p$ if $v$ is early
overcrossing, $t=q$ if $v$ is early undercrossing; $sgn\,v$ denotes
{\em local writhe number} of $v$. We set
${\varepsilon}_1=1\Leftrightarrow$ arc $a$ is emanating from $v$;
 ${\varepsilon}_2=1\Leftrightarrow$ $a$ is passing over $v$;
 ${\varepsilon}_3=1\Leftrightarrow$ $a$ is coming into $v$.
\noindent If $v$ and $a$ are not incident we set $[v:a]=0$.
 }
\end {dfn}

Let us enumerate all classical crossings of $D$ by numbers $1,...,n$
in arbitrary way and associate with each classical crossings the
emanating long arc. Our generalization of Alexander polynomial for
long virtual knots is defined as determinant of $n\times n$-matrix
$A(D)$ with elements

$$A_{ij}:=\sum_{a\subset
{\gamma}^j }\,[v_i:a]s^{deg\,a}\in T[s,s^{-1}]$$
The function $deg:
\{$arcs of D$\} \rightarrow \mathbb{Z}$ is defined according to the
rules:

\noindent (1) if arc $a$ is a first at a long arc, $deg\,a=0$;

\noindent (2) if arcs $a$ and $b$ are neighbour on a long arc, $a$
precedes $b$, then $deg\,b=deg\,a +1$, if we pass from the left to
the right with respect to the transversal arc, and $deg\,b=deg\,a
-1$ otherwise. In the first case we called such virtual crossing
{\em increasing}, in the second case --- {\em decreasing}.

It easy to see that polynomial $\zeta(D)=det\,A(D)$ does not depend
on a numeration of classical crossings.

By analogy with \cite{AM} we formulate following three theorems.

\begin{thm}\label{theorem0}
 If virtual diagrams $D,D^{\prime}$ are equivalent then $\zeta(D^{\prime})=q^{r}\zeta(D)$
 for some integer $r$.
\end{thm}
\begin{proof}[A sketch of the proof]
 The invariance of $\zeta$ for Reidemeister moves
$\Omega_1^{\prime},\Omega_2^{\prime},\Omega_3^{\prime}$ is evident.
The checking of invariance for $\Omega^{\prime}$ and $\Omega_2$
 is similar to the case of $\zeta$-polynomial
for virtual link (see \cite{MyBookEng},\cite{MyBook}).

There are two types of the first Reidemeister move $\Omega_1$:
${\Omega}_1^p$, if we have early overcrossing, and ${\Omega}_1^q$,
if we have early undercrossing. It easy to calculate that
$\zeta({\Omega}_1^p(D))=\zeta(D)$, $\zeta({\Omega}_1^q(D))=q^{\pm
1}\zeta(D)$.

It is convenient to use the Laplace theorem (about determinants) to
check that $det\, A(\Omega_3(D))=det\,A(D)$. We check equality for
$10$ pair of $3\times3$-minors of matrices $A(\Omega_3(D))$ and
$det\,A(D)$. Two of these pairs give equalities only if we set
$(p-1)(p-q)=0,(q-1)(p-q)=0$.

\end{proof}

\begin{thm}\label{theorem1} Let $k$ be the number
of virtual crossings on a long virtual diagram $D$. Then
$deg_s\,\zeta(D)\leqslant k$.
\end{thm}

From Theorems \ref{theorem0} and \ref{theorem1} we easily conclude

\begin{cor}\label{cor1}
If\, $deg_s\,\zeta(D)=k$ then $D$ has minimal virtual crossing
number.
\end{cor}

 For checking of minimality by using Corollary \ref{cor1}\, it is
convenient to use

\begin{thm}\label{theorem3} The $s^k$-th coefficient of\, $\zeta(D)$ is
equal to $det\,B$, where $B_{ij}=[v_i:a_j]$ if\, $\exists\,
a_j\subset{\gamma}^j$ s.t. $deg\,a_j=$\#of increasing virtual
crossings on $\gamma^j$, and $B_{ij}=0$ otherwise, $i,j=1,...,n$.
\end{thm}

{\scshape Example.} In {\bf Figure} we draw long virtual diagram
$D_{r,l}$ which closure is unknot. Arcs $a_j$, $j=1,...,n$, are
marked by thick lines. By Theorem \ref{theorem3} the $s^k$-th
coefficient of\, $\zeta(D)$ is equal to $|[v_i:a_j]|_{i,j=1,...,n}=$
$q^{r+l}(qp^{-1}-1)=q-p\neq 0$ in the ring $T$. Consequently,
$D_{r,l}$ is minimal by Corollary \ref{cor1}.

\begin{figure}
 \centering\includegraphics[scale = 0.6]{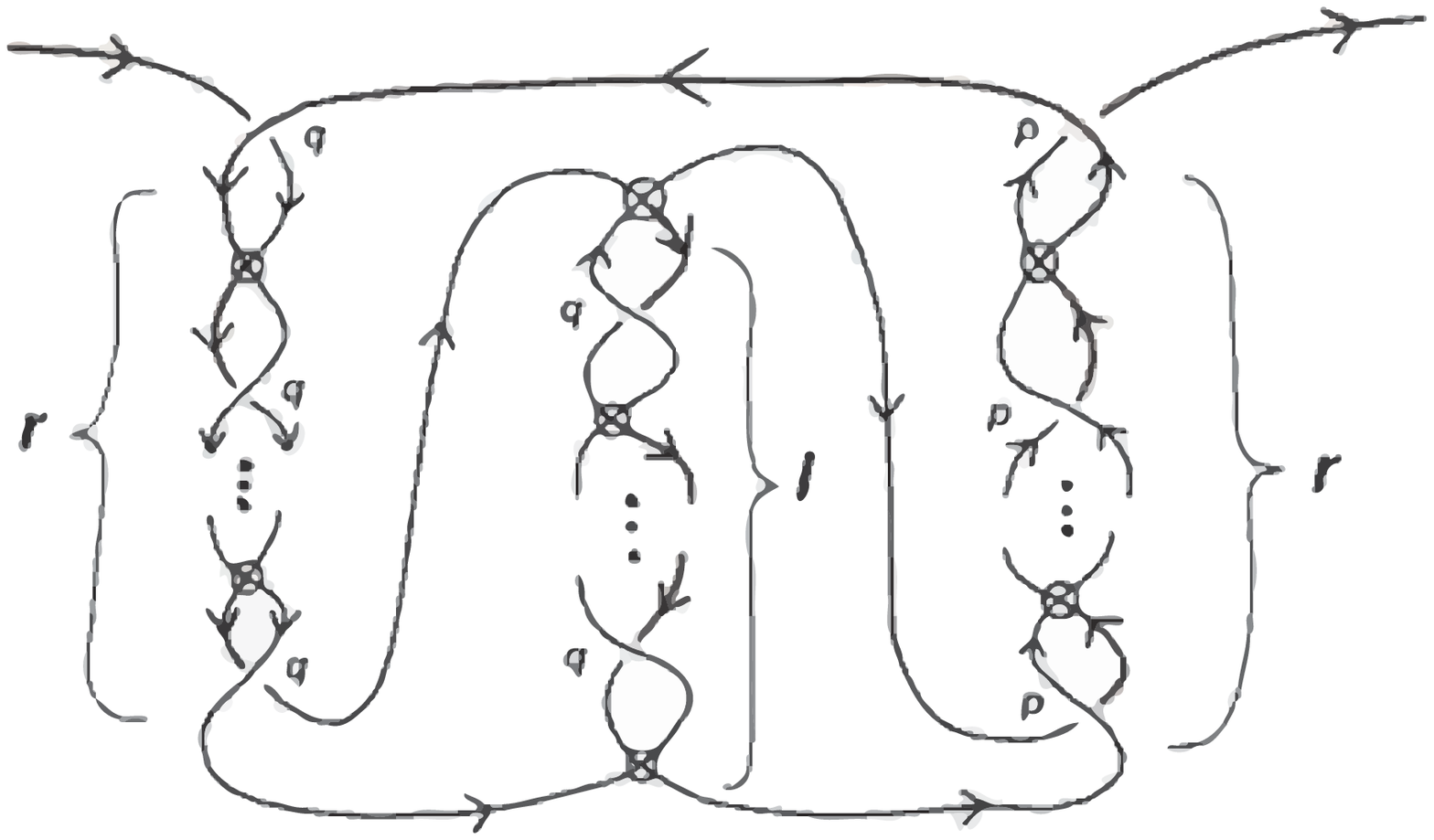}
  \caption{Long knot $D_{r,l}$, $r,l\geqslant0$}
  \label{detour}
 \end{figure}

By using our $\zeta$-polynomial we can proof following Conjecture in
a particular case. Here symbol $*$ denotes usual product of long
knots.

\noindent {\scshape Conjecture.} {\em If $D$ is a minimal long
virtual diagram with respect number of virtual crossings, K is a
long classical knot diagram, then $D*K$ is also minimal.}

\begin{thm}\label{theorem2}(the particular case of Conjecture)\\
If $D$ is a minimal long virtual diagram s.t. $deg_s\,\zeta(D)$ is
equal to virtual crossing number of $D$, $K$ is a long classical
knot diagram, then $D*K$ is minimal.
\end{thm}

For a proof of Theorem \ref{theorem2} we use following lemmas. Let
$l$ be a number of long arc $\gamma={\gamma}_{-}\cup {\gamma}_{+}$,
where ${\gamma}_{-}$ and ${\gamma}_{+}$ are initial and final long
arcs respectively. Then $A_{il}:=\sum_{a\subset \gamma
}\,[v_i:a]s^{deg\,a}=$ $\sum_{a\subset
{\gamma}_{-}}\,[v_i:a]s^{deg\,a}+\sum_{a\subset {\gamma}_{+}
}\,[v_i:a]s^{deg\,a}$. Consequently,
$det\,A(D)=det\,A^{-}(D)+det\,A^{+}(D)$, where
$A^{\pm}_{il}=\sum_{a\subset{\gamma}_{\pm}}\,[v_i:a]s^{deg\,a}$,
$A^{\pm}_{ij}=A_{ij}$ for $j\neq l$. Thus, we have the natural
decomposition of $\zeta$-polynomial:
$\zeta(D)=\zeta_{-}(D)+\zeta_{+}(D)$, where
$\zeta_{\pm}(D):=det\,A^{\pm}(D)$.

\begin{lem}\label{lem1}
  $\zeta_{-}(D_1*D_2)=-\zeta_{-}(D_1)\zeta_{-}(D_2)$;
   $\zeta_{+}(D_1*D_2)=\zeta_{+}(D_1)\zeta_{+}(D_2)$.
\end{lem}

\begin{lem}\label{lem2}
  $x\in T=\mathbb{Z}[p,p^{-1},q,q^{-1}]/((p-1)(p-q),(q-1)(p-q))$ is
  zero divisor $\Leftrightarrow$ $x|_{p=1,\,q=1}=0$.
\end{lem}

\begin{proof}[Proof of Theorem \ref{theorem2}]

By Lemma \ref{lem1}\, $\zeta(D*K)=\zeta_{-}(D*K)+\zeta_{+}(D*K)=$
$-\zeta_{-}(D)\zeta_{-}(K)+\zeta_{+}(D)\zeta_{+}(K)=$
$\zeta_{+}(K)\zeta(D)$, because $\zeta(K)=0$. Consequently,
$deg_s\,\zeta(D*K)=deg_s\,\zeta(D)$ if $\zeta_{+}(K)\in T$ is not
zero divisor.

It easy to check that $\zeta_{+}(K)|_{p=1,\,q=1}=\pm
\Delta(K)|_{t=1}$, where $\Delta$ denotes Alexander polynomial. It
is known that $\Delta(K)|_{t=1}=\pm 1$. Hence, by Lemma \ref{lem2}\,
$\zeta_{+}(K)$ is not zero divisor, because
$\zeta_{+}(K)|_{p=1,\,q=1}\neq 0$.
\end{proof}

\section*{Acknowledgements}

The author is grateful to V.O. Manturov for idea of
$\zeta$-polynomial for long virtual knots and fruitful
consultations.


\begin{thebibliography}{100}

\bibitem[AM]{AM} D.M.~Afanasiev, V.O.~Manturov, On Virtual Crossing Number Estimates
For Virtual Links, {\em Journal of Knot Theory and Its
Ramifications}, Vol. 18, No. 6 (2009).

\bibitem[BF]{BF} A.\,~Bartholemew, R.\,~Fenn (2003), Quaternionic
invariants of virtual knots and links, {\em Journal of Knot Theory
and Its Ramifications}, {\bf  17} (2),2008 pp. 231-251

 \bibitem[GPV]{GPV}
M.~Goussarov, M.~Polyak, and O.~Viro, Finite type invariants of
classical and virtual knots, Topology, 2000, V.\, 39, pp.\,
1045--1068.

\bibitem[Ka1]{KaV}
L.\,H.~Kauffman, Virtual knot theory, Eur. J. Combinatorics. 1999.
V.\, 20, N.\, 7, pp.\, 662--690.

\bibitem[KM]{KM} L.H.~Kauffman, V.O.~Manturov, Virtual biquandles,
Fundamenta Mathematicae 188 (2005), pp. 103-146.

\bibitem[KR]{KR} L.H.Kauffman, D.Radford (2002), Bi-oriented quantum
algebras and a generalized Alexander polynomial for virtual links,
{\em AMS Contemp. Math.}, {\bf 318}, pp. 113-140.

\bibitem[Ma1]{MaXi}
 V.O.~Manturov, An Invariant 2-variable polynomial for virtual links (2002),
(Russian Math. Surveys), {\bf 57}, No.5, P.141-142.

\bibitem[Ma2] {MyBookEng} V.O.~Manturov, Knot Theory, Chapman \& Hall, London,
CRC Press.

 \bibitem[Ma3]{MyBook}
V.O.~Manturov, {\em Teoriya Uzlov} (Knot Theory), (Moscow-Izhevsk,
RCD), 2005 (in Russian).

\bibitem[Ma4]{MaL} Long virtual knots and their invariants, ibid.
13 (2004), 1029-1039.

\bibitem[Saw]{Saw} J. Sawollek (2002), On Alexander-Conway Polynomials for
Virtual Knots and Links, {\em J. Knot Theory and Its Ramifications},
{\bf 12} (6), pp.767-779.

\bibitem[SW]{SW} D.Silver and S.Williams (2001), Alexander Groups
and Virtual Links, {\em J. of Knot Theory and Its Ramifications},
{\bf 10} (1), pp. 151-160.

\end{thebibliography}
\end{document}